\documentclass{article}
\usepackage{amssymb}
\usepackage{amsthm}
\usepackage[curve,matrix,arrow]{xy}
\theoremstyle{plain}\newtheorem{Theorem}{Theorem}[section]
\theoremstyle{plain}
\theoremstyle{plain}\newtheorem{Corollary}[Theorem]{Corollary}
\theoremstyle{plain}\newtheorem{Lemma}[Theorem]{Lemma}
\theoremstyle{plain}\newtheorem{Proposition}[Theorem]{Proposition}
\theoremstyle{definition}\newtheorem{Definition}[Theorem]{Definition}
\theoremstyle{definition}
\theoremstyle{definition}
\theoremstyle{definition}\newtheorem{Remark}[Theorem]{Remark}
\theoremstyle{definition}
\theoremstyle{plain}

    \def\OG{{\mathcal{O}G}}  \def\OGb{{\mathcal{O}Gb}}
    \def\OH{{\mathcal{O}H}}  \def\OHc{{\mathcal{O}Hc}}
    \def\OP{{\mathcal{O}P}}
    \def\OQ{{\mathcal{O}Q}}
    \def\OR{{\mathcal{O}R}}
\def\CF{{\mathcal{F}}}    
                          \def\OT{{\mathcal{O}T}}

\def\CO{{\mathcal{O}}}

\def\Aut{\mathrm{Aut}}           \def\tenk{\otimes_k}     
\def\Br{\mathrm{Br}}             \def\ten{\otimes}

\def\Defres{\mathrm{Defres}}
           
\def\End{\mathrm{End}}

\def\Hom{\mathrm{Hom}}

\def\Id{\mathrm{Id}}             \def\tenA{\otimes_A}
             \def\tenB{\otimes_B}
\def\Ind{\mathrm{Ind}}

           \def\tenO{\otimes_{\mathcal{O}}}
      
\def\op{\mathrm{op}} 
    
           \def\tenOP{\otimes_{\mathcal{O}P}}
         \def\tenOQ{\otimes_{\mathcal{O}Q}}
\def\Res{\mathrm{Res}}           \def\tenOR{\otimes_{\mathcal{O}R}}
         
\def\soc{\mathrm{soc}}           \def\tenOS{\otimes_{\mathcal{O}S}}
             \def\tenOT{\otimes_{\mathcal{O}T}}
             
                                 \def\tenOH{\otimes_{\mathcal{O}H}}

\def\tenOGb{\otimes_{\mathcal{O}Gb}}
\def\tenOHc{\otimes_{\mathcal{O}Hc}}

\title{On stable equivalences with endopermutation source} 
\author{Markus Linckelmann} 
\date{}

\begin{document}

\maketitle

\begin{abstract}
We show that a  bimodule between block algebras which has a 
fusion stable endopermutation module as a source and which induces
Morita equivalences between centralisers of nontrivial subgroups
of a defect group induces a stable equivalence of Morita type; this
is a converse to a theorem of Puig. 
The special case where the source is trivial has long been known
by many authors. The earliest instance for a result deducing 
a stable equivalence of Morita type from local Morita equivalences
with possibly nontrivial endopermutation source is due to Puig,
in the context of blocks with abelian defect groups with 
a Frobenius inertial quotient. The present note is motivated by
an application, due to Biland, to blocks of finite groups with
structural properties known to hold for hypothetical 
minimal counterexamples to the $Z_p^*$-Theorem.
\end{abstract}


\section{Introduction}

Let $p$ be a prime and $\CO$ a complete discrete valuation ring having 
a residue field $k$ of characteristic $p$; we allow the case $\CO=$ $k$.  
We will assume that $k$ is a splitting field for all block algebras which 
arise in this note. Following Brou\'e \cite[\S 5.A]{BroueEq}, given two 
$\CO$-algebras $A$, $B$, an $A$-$B$-bimodule $M$ and a $B$-$A$-bimodule $N$, 
we say that {\it $M$ and $N$ induce a stable equivalence of Morita type 
between $A$ and $B$} if $M$, $N$ are finitely generated projective as left 
and right modules, and if $M\tenB N\cong$ $A\oplus W$ for some projective 
$A\tenO A^\op$-module $W$ and $N\tenA M\cong$ $B\oplus$ $W'$ for some 
projective $B\tenO B^\op$-module $W'$.
By a result of Puig in \cite[7.7.4]{Puigbook} a stable  equivalence of 
Morita type between block algebras of finite groups given by a bimodule 
with endopermutation source and its dual implies that there is a canonical 
identification of the defect groups of the two blocks such that both have 
the same local strucure and such that  corresponding blocks of centralisers 
of nontrivial subgroups of that  common defect group are Morita equivalent 
via bimodules with endopermutation sources. The following theorem is a 
converse to this result. The terminology and required background information 
for this statement are collected in the next two sections, together
with further references.

\begin{Theorem} \label{endopermutationstableMorita}
Let $A$, $B$ be almost source algebras of blocks of finite group
algebras over $\CO$ having a common defect group $P$ and
the same fusion system $\CF$ on $P$.
Let $V$ be an $\CF$-stable indecomposable endopermutation $\OP$-module
with vertex $P$, viewed as an $\CO\Delta P$-module
through the canonical isomorphism $\Delta P\cong$ $P$.
Let $M$ be an indecomposable direct summand of the
$A$-$B$-bimodule
$$A\tenOP\Ind^{P\times P}_{\Delta P}(V)\tenOP B\ .$$
Suppose that $(M\tenB M^*)(\Delta P)\neq$ $\{0\}$. Then
for any nontrivial fully $\CF$-centralised subgroup $Q$ of $P$,
there is a canonical $A(\Delta Q)$-$B(\Delta Q)$-bimodule $M_Q$ satisfying
$\End_k(M_Q)\cong$ $(\End_\CO(M))(\Delta Q)$. Moreover, if for all
nontrivial fully $\CF$-centralised subgroups $Q$ of $P$ the bimodule
$M_Q$ induces a Morita equivalence between $A(\Delta Q)$ and $B(\Delta Q)$,
then $M$ and its dual $M^*$ induce a stable equivalence
of Morita type between $A$ and $B$.
\end{Theorem}

For $V$ the trivial $\OP$-module, variations of the above result have 
been noted by many authors. For principal blocks this was first pointed 
out by Alperin. A version for finite groups with the same local structure 
appears in Brou\'e \cite[6.3]{BroueEq}, and the above theorem with $V$ 
trivial is equivalent to \cite[Theorem 3.1]{Lisplendid}. 
The first class of examples for this situation with potentially nontrivial 
$V$ goes back to work of Puig \cite{Puabelien}: it is shown in 
\cite[6.8]{Puabelien} that a block with an abelian defect group $P$ and a 
Frobenius inertial quotient is stably equivalent to its Brauer 
correspondent, using the fact that the blocks of centralisers of 
nontrivial subgroups of $P$ are nilpotent, hence Morita equivalent to the 
defect group algebra via a Morita equivalence with endopermutation source. 
The above theorem is used in the proof of Biland 
\cite[Theorem 4.1]{Biland} or \cite[Theorem 1]{BilandZp}.
For convenience, we reformulate this at the block algebra level. 

\begin{Theorem} \label{blockendopermutationstableMorita}
Let $G$, $H$ be finite groups, and let $b$, $c$ be blocks of $\OG$, $\OH$,
respectively, having a common defect group $P$. Let $i\in$ $(\OGb)^{\Delta P}$
and $j\in$ $(\OHc)^{\Delta P}$ be almost source idempotents. For any subgroup
$Q$ of $P$ denote by $e_Q$ and $f_Q$ the unique blocks of $kC_G(Q)$ and
$kC_H(Q)$, respectively, satisfying $\Br_{\Delta Q}(i)e_Q\neq$ $0$ and 
$\Br_{\Delta Q}(j)f_Q\neq$ $0$. Denote by $\hat e_Q$ and $\hat f_Q$ the
unique blocks of $\CO C_G(Q)$ and $\CO C_H(Q)$ lifting $e_Q$ and $f_Q$,
respectively.
Suppose that $i$ and $j$ determine the same fusion system $\CF$ on $P$.
Let $V$ be an $\CF$-stable indecomposable endopermutation $\OP$-module
with vertex $P$, viewed as an $\CO\Delta P$-module
through the canonical isomorphism $\Delta P\cong$ $P$.
Let $M$ be an indecomposable direct summand of the
$\OGb$-$\OHc$-bimodule
$$\OG i\tenOP\Ind^{P\times P}_{\Delta P}(V)\tenOP j\OH\ .$$
Suppose that $M$ has $\Delta P$ as a vertex as an $\CO(G\times H)$-module. 
Then for any nontrivial subgroup $Q$ of $P$, there is a canonical 
$kC_G(Q)e_Q$-$kC_H(Q)f_Q$-bimodule $M_Q$ satisfying $\End_k(M_Q)\cong$ 
$(\End_\CO(\hat e_QM\hat f_Q))(\Delta Q)$. Moreover, if for all nontrivial 
subgroups $Q$ of $P$ the bimodule $M_Q$ induces a Morita equivalence 
between $kC_G(Q)e_Q$ and $kC_H(Q)f_Q$, then $M$ and its dual $M^*$ 
induce a stable equivalence of Morita type between $\OGb$ and $\OHc$.
\end{Theorem}

The existence of canonical bimodules $M_Q$ satisfying $\End_k(M_Q)\cong$ 
$(\End_\CO(\hat e_QM\hat f_Q))(\Delta Q)$ in this Theorem is due to Biland
\cite[Theorem 3.15]{Biland}. In the statement of Theorem 
\ref{blockendopermutationstableMorita} we let $Q$ run over all nontrivial
subgroups of $P$ rather than only the fully $\CF$-centralised ones; this
makes no difference here since one can always achieve $Q$ to be fully
centralised through simultaneous conjugation in $G$ and $H$. By contrast,
in the statement of Theorem \ref{endopermutationstableMorita}, restricting
attention to fully centralised subgroups is necessary in order to ensure
that $A(\Delta Q)$ and $kC_G(Q)e_Q$ are Morita equivalent. Another technical
difference between the statements of the two theorems is that
$iMj$ will be an endopermutation $\CO\Delta Q$-module, while this is not clear
for $\hat e_QM\hat f_Q$ because indecomposable $\CO\Delta Q$-summands with
vertices strictly smaller than $\Delta Q$ might not be compatible. 
See Biland \cite[Lemma 10]{Bilandfriendly} for more details on this issue.

\begin{Remark} 
The proof of Theorem \ref{endopermutationstableMorita} becomes
significantly shorter if one assumes that $V$ has a 
fusion-stable endosplit $p$-permutation resolution. This concept is due 
to Rickard \cite{Ricksplendid}, who also showed the existence of such 
resolutions for finite abelian $p$-groups.  
As a consequence of the classification of endopermutation modules, 
endosplit $p$-permutation resolutions exist for all 
endopermutation modules over $k$ and unramified $\CO$ belonging to
the subgroup of the Dade group generated by relative syzygies. For odd
$p$ this is the entire Dade group while for $p=2$ there are some
endopermutation modules which do not have endosplit permutation
resolutions. See \cite[Theorem 14.3]{Thevtour} for more details.  
We will outline how this simplifies the proof
in the Remark \ref{endosplitRemark} below.
\end{Remark}

\section{Background material on blocks and almost source algebras}

Let $G$ be a finite group. For any subgroup $H$ of $G$, we denote by 
$\Delta H$ the `diagonal' subgroup $\Delta H=$ $\{(y,y)\ |\ y\in H\}$ of 
$H\times H$. Let $b$ be a block of $\OG$ and $P$ a defect group of $b$.
That is, $b$ is a primitive idempotent in $Z(\OG)$, and $P$ is a maximal
$p$-subgroup with the property that $\OP$ is isomorphic to a direct 
summand of $\OGb$ as an $\OP$-$\OP$-bimodule. As is customary, for any 
$p$-subgroup $Q$ of $G$ we denote by  
$\Br_{\Delta Q} : (\OG)^{\Delta Q}\to$ $kC_G(Q)$ the {\it Brauer 
homomorphism} induced by the linear map sending $x\in$ $C_G(Q)$ to its 
image in $kC_G(Q)$ and $x\in$ $G\setminus C_G(Q)$ to zero. The map 
$\Br_{\Delta Q}$ is a surjective algebra homomorphism. More generally, 
for $Q$ a $p$-subgroup of $G$ and $M$ an $\OG$-module, we denote by 
$M(Q)$ the $kN_G(Q)$-module obtained from applying the Brauer 
construction $\Br_Q$  to $M$. If $A$ is an interior $P$-algebra, and $Q$ 
a subgroup of $Q$, we denote by $A(\Delta P)$ the interior  
$C_P(Q)$-algebra obtained from applying the Brauer construction with 
respect to the conjugation action of $P$ on $A$; our notational 
conventions are as in \cite[\S 3]{Liperm}.  

\medskip
Following Green \cite{Greenindec}, for any indecomposable $\OGb$-module $U$, 
if $Q$ is a minimal subgroup of $G$ for which there exists an $\OQ$-module $V$ such
that $U$ is isomorphic to a direct summand of $\Ind^G_Q(V)$, then $Q$ is a
$p$-subgroup of $G$, the $\OQ$-module $V$ can
be chosen to be indecomposable, in which case $V$ is isomorphic to a direct
summand of $\Res^G_Q(U)$, and the pair $(Q,V)$ is unique up to $G$-conjugacy.
In that situation, $Q$ is called a vertex of $U$, and $V$ an $\OQ$-source
of $U$, or simply a source of $U$ of $Q$ is determined by the context. 
Moreover, if $R$ is a $p$-subgroup of $G$ such that $\Res^G_R(U)$ has an
indecomposable direct summand $W$ with vertex $R$, then there is
a vertex-source pair $(Q,V)$ of $U$ such that $R\subseteq$ $Q$ and such
that $W$ is isomorphic to a direct summand of $\Res^Q_R(V)$. 
By Higman's criterion, this happens if and only of $(\End_\CO(U))(\Delta R)\neq$
$\{0\}$. See \cite[Chapter 4]{NaTs} for an exposition of Green's theory of vertices 
and sources.

\begin{Definition} [{cf. \cite[Definition 4.3]{Liperm}}]
Let $G$ be a finite group, let $b$ be a block of $\OG$, let $P$ be a defect
group of $b$. An idempotent $i$ in $(\OGb)^{\Delta P}$ is called an
{\it almost source idempotent} if $\Br_{\Delta P}(i)\neq 0$
and for every subgroup $Q$ of $P$ there is a unique block $e_Q$
of $kC_G(Q)$ such that $\Br_{\Delta Q}(i)\in kC_G(Q)e_Q$.
The interior $P$-algebra $i\OG i$ is then called an {\it almost
source algebra of the block} $b$.
\end{Definition}

By \cite[3.5]{Puigpoint} (see also \cite[Proposition 4.1]{Liperm} for a
proof) there is a canonical Morita equivalence between the block algebra 
$\OGb$ and an almost source algebra $i\OG i$ sending an $\OGb$-module $M$ 
to the $i\OG i$-module $iM$. 
Regarding fusion systems, we tend to follow the conventions of
\cite[\S 2]{LinSol}; in particular, by a fusion system on a finite 
$p$-group we always mean a saturated fusion system (in the 
terminology used in \cite{AKO} or \cite{Craven}, for instance).  
With the notation of the previous Definition, it follows from work of 
Alperin and Brou\'e \cite{AlBr} that the choice of an almost source 
idempotent $i$ in $(\OGb)^{\Delta P}$ determines a fusion system $\CF$ 
on $P$ such that for any two subgroups $Q$, $R$ of $P$, the set 
$\Hom_\CF(Q,R)$ is the set of all group homomorphisms $\varphi : Q \to$ 
$R$ for which there is an element $x\in G$ satisfying $\varphi(u) =$ 
$xux^{-1}$ for all $u\in Q$ and satisfying $xe_Qx^{-1} =$ 
$e_{xQx^{-1}}$. See e. g. \cite[\S 2]{LinSol}, or \cite[Part IV]{AKO};
note that we use here our blanket assumption that $k$ is large enough.
Moreover, a subgroup $Q$ of $P$ is fully $\CF$-centralised if
and only if $C_P(Q)$ is a defect group of the block $e_Q$ of $kC_G(Q)$.
Given a subgroup $Q$ of $P$ it is always possible to find
a subgroup $R$ of $P$ such that $Q\cong R$ in $\CF$ and such that $R$ is
fully $\CF$-centralised. 

\begin{Proposition} [{cf. {\cite[Proposition 4.5]{Liperm}}}]
Let $G$ be a finite group, $b$ a block of $\OG$, $P$ a defect group of
$b$, and $i\in$ $(\OGb)^{\Delta P}$ an almost source idempotent of $b$.
with associated almost source algebra $A=$ $i\OG i$. If $Q$ is a fully 
$\CF$-centralised subgroup of $P$, then $\Br_{\Delta Q}(i)$ is an almost 
source idempotent of $kC_G(Q)e_Q$ with associated almost source algebra 
$A(\Delta Q)$; in particular, $kC_G(Q)e_Q$ and $A(\Delta Q)$ are Morita 
equivalent.
\end{Proposition}

By \cite[4.2]{Liperm}, an almost source algebra $A$ of a block with $P$
as a defect group is isomorphic to a direct summand of $A\tenOP A$ as an 
$A$-$A$-bimodule. Since $A\tenOP A\cong$ 
$(A\tenO A^\op)\ten_{\OP\ten\OP^\op}\ \OP$
this means that as an $A\tenO A^\op$-module, $A$ is relatively 
$\OP\tenO \OP^\op$-projective. Since $A$, $\OP$, and hence $A\tenO A^\op$,
$\OP\tenO \OP^\op$ are symmetric $\CO$-algebras, it follows that
$A$ is also relatively $\OP\tenO \OP^\op$-injective.
Tensoring a split map $A\to$ $A\tenOP A$ by $-\tenA U$ implies that
any $A$-module $U$ is relatively $\OP$-projective,
or equivalently, isomorphic to a direct summand of $A\tenOP U$. 
Vertices and sources of indecomposable 
$\OGb$-modules can be read off from almost  source algebras; the following 
result is a slight generalisation of \cite[6.3]{Likleinfour}. 

\begin{Proposition} \label{vertexsourcealgebra}
Let $G$ be a finite group, $b$ a block of $\OG$, $P$ a defect group
of $b$, and $i$ an almost source idempotent in $(\OGb)^{\Delta P}$. Set
$A=$ $i\OG i$. Let $U$ be an indecomposable $\OGb$-module, and let $Q$
be a minimal subgroup of $P$ such that the $A$-module $iU$ is isomorphic 
to a direct summand of $A\tenOQ V$ for some $\OQ$-module $V$. Then $Q$ 
is a vertex of $U$, and $U$ is isomorphic to a direct summand of 
$\OG i\tenOQ iU$, or equivalently, the $\OQ$-module $V$ with the property 
that $iU$ is isomorphic to a direct summand of $A\tenOQ V$ can be chosen 
to be an indecomposable direct summand of $\Res_Q(iU)$. 
\end{Proposition}

\begin{proof}
Note that $iU$ is an indecomposable $A$-module.
Let $Q$ be a minimal subgroup of $P$ such that $iU$ is isomorphic to a 
direct summand of $A\tenOQ V$, for some $\OQ$-module $V$. Tensoring 
with $\OG i\tenA -$ implies that $U$ is isomorphic to a direct summmand 
of $\OG i\tenOQ V$, hence of $\Ind^G_Q(V)$. Thus $Q$ contains a vertex 
of $U$. By general abstract nonsense (e.g. the equivalence of the 
statements (i) and (ii) in \cite[Theorem 6.8]{BroueHigman} applied to 
restriction and induction between $A$ and $\OQ$), $iU$ is then 
isomorphic to a direct summand of $A\tenOQ iU$, thus of $A\tenOQ V$ for 
some indecomposable direct summand $V$ of $\Res_Q(iU)$. The minimality 
of $Q$ implies that $V$ has $Q$ as a vertex. But $V$ is isomorphic to a 
direct summand of $\Res^G_Q(U)$, and hence $Q$ is contained
in a vertex of $U$. The result follows.
\end{proof}

By a result of Puig in \cite{Pulocsource}, fusion systems of blocks can 
be read off their source algebras; this is slightly extended to almost 
source algebras in \cite[5.1, 5.2]{Liperm}. 

\begin{Proposition} [{cf. \cite[Proposition 5.1]{Liperm}}] 
\label{almostsourcealgebrabimodfusion}
Let $G$ be a finite group, let $b$ be a block of $\OG$ with defect group
$P$, let $i\in (\OGb)^{\Delta P}$ be an almost source idempotent and set
$A = i\OG i$. Denote by $\CF$ the fusion system of $A$ on $P$. Let $Q$ be
a fully $\CF$-centralised subgroup of $P$ and let $\varphi : Q \rightarrow P$
be a morphism in $\CF$. Set $R = \varphi(Q)$. Denote by $e_Q$, $e_R$ the
unique blocks of $kC_G(Q)$, $kC_G(R)$ satisfying
$\Br_{\Delta Q}(i)e_Q \neq 0$ and $\Br_{\Delta R}(i)e_R \neq 0$.

\smallskip\noindent
(i) For any primitive idempotent $n$ in $(\OGb)^{\Delta R}$
satisfying $\Br_{\Delta R}(n)e_R \neq 0$ there is a primitive idempotent
$m$ in $A^{\Delta Q}$ satisfying $\Br_{\Delta Q}(m) \neq 0$
such that $m\OG \cong {_\varphi(n\OG)}$ as $\OQ$-$\OGb$-bimodules
and such that $\OG m \cong (\OG n)_\varphi$ as $\OGb$-$\OQ$-bimodules.

\smallskip\noindent
(ii) For any primitive idempotent $n$ in $A^{\Delta R}$ satisfying
$\Br_{\Delta R}(n) \neq 0$ there is a primitive idempotent $m$ in
$A^{\Delta Q}$ satisfying $\Br_{\Delta Q}(m) \neq 0$
such that $mA \cong {_\varphi(nA)}$ as $\OQ$-$A$-bimodules
and such that $Am \cong (An)_\varphi$ as $A$-$\OQ$-bimodules.
\end{Proposition}

\begin{Proposition} [{cf. \cite[Proposition 5.2]{Liperm}}]
\label{almostsourcealgebrafusion}
Let $G$ be a finite group, $b$ be a block of $\OG$ with defect
group $P$, let $i$ be an almost source idempotent in $(\OGb)^{\Delta P}$
and set $A = i\OG i$. Denote by $\CF$ the fusion system of $A$ on $P$.
Let $Q$, $R$ be subgroups of $P$.

\smallskip\noindent
(i) Every indecomposable direct summand of $A$ as an
$\OQ$-$\OR$-bimodule is isomorphic to $\OQ\tenOS {_\varphi\OR}$
for some subgroup $S$ of $Q$ and some morphism $\varphi : S \to$
$R$ belonging to $\CF$.

\smallskip\noindent
(ii) If $\varphi : Q \to R$ is an isomorphism in $\CF$
such that $R$ is fully $\CF$-centralised then
${_\varphi\OR}$ is isomorphic to a direct summand of $A$
as an $\OQ$-$\OR$-bimodule.

\smallskip\noindent
In particular, $\CF$ is determined by the $\OP$-$\OP$-bimodule
structure of $A$.
\end{Proposition}

\begin{Proposition} \label{Brauerpairfusion}
Let $G$ be a finite group, $b$ be a block of $kG$ with defect
group $P$, and let $i$ be an almost source idempotent in $(kGb)^{\Delta P}$.
Denote by $\CF$ the fusion system on $P$ determined by $i$.
Let $Q$, $R$ be subgroups of $P$ and denote by $e$ the unique block
of $kC_G(Q)$ satisfying $\Br_{\Delta Q}(i)e\neq$ $0$. 
Let $\varphi : Q\to$ $R$ be an injective group homomorphism such that
${_\varphi{kR}}$ is isomorphic to a direct summand of $e kG i$ as
an $kQ$-$kR$-bimodule. Then $\varphi\in$ $\Hom_\CF(Q,R)$.
\end{Proposition}

\begin{proof}
Let $T$ be a fully $\CF$-centralised subgroup of $P$ isomorphic to
$Q$ in the fusion system $\CF$. That is, if $f$ is the unique
block of $kC_G(T)$ satisfying $\Br_{\Delta T}(i)f\neq$ $0$, then
$C_P(T)$ is a defect group of $kC_G(T)f$, there is an element
$x\in$ $G$ such that $(T,f)=$ ${^x{(Q,e)}}$, and the isomorphism
$\psi : Q\to T$ defined by $\psi(u)=$ $xux^{-1}$ is in $\CF$.
Since ${_\varphi{kR}}$ is a summand of $e_QkG i$, multiplication
by $x$ shows that the $kT$-$kR$-bimodule  
${_{\varphi\circ\psi^{-1}}{kS}}$ is a direct summand of $xekGi=$
$xex^{-1}kGi=$ $fkGi$. Moreover, $\varphi$ is a morphism in $\CF$ if and only
if $\varphi\circ\psi^{-1}$ is. Thus, after possibly replacing
$(Q,e)$ by $(T,f)$ we may assume that $(Q,e)$ is fully $\CF$-centralised.
By \cite[Proposition 4.6]{Liperm} (ii)
this implies that every local point of $Q$ on $kGb$ associated with $e$
has a representative in $ikGi$. Since ${_\varphi{kR}}$ is indecomposable
as a $kQ$-$kR$-bimodule with a vertex of order $|Q|$, this bimodule is
isomorphic to a direct summand of $jkGi$ for some primitive local idempotent
$j$ in $(kGb)^{\Delta Q}$ appearing in a primitive decomposition of 
$e$ in $(kGb)^{\Delta Q}$. But then $\Br_{\Delta Q}(j)\in$ $kC_G(Q)e$, and
hence, after possibly replacing $j$ with a suitable 
$((kGb)^{\Delta Q})^\times$-conjugate, we may assume that $j\in$ $ikGi$.
It follows from Proposition \ref{almostsourcealgebrafusion} (i) that
$\varphi$ is a morphism in $\CF$.
\end{proof}

\begin{Proposition} \label{diagbimodtensorproduct}
Let $G$ be a finite group, $H$ a subgroup of $G$ and $A$ an $\CO$-algebra.
Let $M$ be an $\OH$-$A$-bimodule and $V$ an $\OH$-module.
Consider $V\tenO M$ as an $\OH$-$A$ bimodule with $H$
acting diagonally on the left, consider $V$ as a module
for $k\Delta H$ via the canonical isomorphism $\Delta H\cong$
$H$ and consider $\Ind^{G\times H}_{\Delta H}(V)$ as
an $\OG$-$\OH$-bimodule. We have a natural isomorphism
of $\OG$-$A$-bimodules
$$\Ind^G_H(V\tenO M)\cong \Ind^{G\times H}_{\Delta H}(V)\tenOH M$$
sending $x\ten(v\ten m)$ to $((x,1)\ten v)\ten m$, where
$v\in$ $V$ and $m\in$ $M$.
\end{Proposition}

\begin{proof} This is a straightforward verification.
\end{proof}

\section{On fusion-stable endopermutation modules}

Let $P$ be a finite $p$-group. Following Dade \cite{Dadeendo} a finitely 
generated $\CO$-free $\OP$-module $V$ is an {\it endopermutation module} 
if $\End_\CO(V)\cong$ $V\tenO V^*$ is a permutation $\OP$-module, with
respect to the `diagonal' action of $P$. See Th\'evenaz \cite{Thevtour} 
for an overview on this subject and some historic background, leading up 
to the classification of endopermutation modules.  We will use without 
further comment some of the basic properties, due to Dade, of 
endopermutation modules - see for instance \cite[\S 28]{Thev}. 
If $V$ is an endopermutation $\OP$-module having an indecomposable 
direct summand with vertex $P$, then for any two subgroups $Q$, $R$ of 
$P$ such that $Q$ is normal in $R$, there is an endopermutation 
$kR/Q$-module $V'=$ $\Defres^P_{R/Q}(V)$ satisfying 
$\End_\CO(V)(\Delta Q)\cong$ $\End_k(V')$ as $R/Q$-algebras, and
as interior $R/Q$-algebras if $R\subseteq$ $QC_P(Q)$. This construction
is also known as Dade's `slash' construction.

\begin{Definition} \label{fusionstablemodule}
Let $P$ be a finite $p$-group and $\CF$ a fusion system on $P$. Let $Q$
be a subgroup of $P$ and $V$ be an endopermutation $\OQ$-module. We say
that $V$ is {\it $\CF$-stable} if for any subgroup $R$ of $Q$ and any
morphism $\varphi : R\to$ $Q$ in $\CF$ the sets of isomorphism classes of
indecomposable direct summands with vertex $R$ of the $\OQ$-modules
$\Res^{Q}_R(V)$ and ${_\varphi{V}}$ are equal (including the possibility
that both sets may be empty).
\end{Definition}

With the notation of \ref{fusionstablemodule}, the property of $V$ being
$\CF$-stable does not necessarily imply that $\Res^{Q}_R(V)$ and
${_\varphi{V}}$ have to be {\it isomorphic} as $\OR$-modules, where
$\varphi : R\to$ $Q$ is a morphism in $\CF$ (so this is a slight deviation
from the terminology in \cite[3.3. (1)]{Lima}). What the $\CF$-stability of
$V$  means is that the indecomposable direct factors of $\Res^{Q}_R(V)$
and ${_\varphi{V}}$ with vertex $R$, if any, are isomorphic, but they may
occur with different multiplicities in direct sum decompositions (in other 
words, in the terminology of \cite[3.3.(2)]{Lima} the class of $V$ in the 
Dade group is $\CF$-stable, provided that $V$ has an indecomposable direct 
summand with vertex $P$). By \cite[3.7]{Lima}, every class in $D_\CO(P)$ 
having an $\CF$-stable representative has a representative $W$ satisfying 
the stronger stability condition $\Res^P_R(W)\cong$ ${_\varphi{W}}$ for 
any morphism $\varphi : R\to$ $P$ in $\CF$. It follows from Alperin's 
fusion theorem that in order to check whether an 
endopermutation $\OP$-module $V$ with an indecomposable direct summand 
of vertex $P$ is $\CF$-stable, it suffices to verify that $\Res^P_R(V)$ and
${_\varphi{V}}$ have isomorphic summands with vertex $R$ for any
$\CF$-essential subgroup $R$ of $P$ and any $p'$-automorphism $\varphi$
of $R$ in $\Aut_\CF(R)$. In particular, if $P$ is abelian, then an
indecomposable endopermutation $\OP$-module $V$ with vertex $P$ is
$\CF$-stable if and only if $V\cong$ ${_\varphi{V}}$ for any $\varphi\in$
$\Aut_\CF(P)$. In the majority of cases where Definition
\ref{fusionstablemodule} is used we will have $Q=$ $P$. One notable
exception arises in the context of bimodules, where we consider the
fusion systems $\CF\times\CF$ on $P\times P$ with the diagonal subgroup
$\Delta P$ playing the role of $Q$. The key argument exploiting the
$\CF$-stability of an endopermutation $\OP$-module $V$ having an
indecomposable direct summand with vertex $P$ goes as follows: if $Q$ is
a subgroup of $P$ and $\varphi : Q\to$ $P$ a morphism in $\CF$, then the
restriction to $\Delta Q$ of $V\tenO {_\varphi{V^*}}$ is again a
permutation module, or equivalently, $V\tenO V^*$ remains a permutation
module for the twisted diagonal subgroup $\Delta_\varphi Q=$
$\{(u,\varphi(u))\ |\ u\in Q\}$ of $P\times P$.
By a result of Brou\'e in \cite{Broue85},
if $V$ is a permutation $\OP$-module, then $\End_\CO(V)(\Delta P)\cong$
$\End_k(V(P))$. This is not true for more general modules, but Dade's
`slash' construction from \cite{Dadeendo} for endopermutation modules 
yields a generalisation of this isomorphism, as follows. 

\begin{Proposition} \label{endoinducedsource}
Let $A$ be an almost source algebra of a block of a finite group algebra
over $\CO$ with a defect group $P$ and fusion system $\CF$ on $P$.
Let $Q$ be a subgroup of $P$ and let $V$ be an $\CF$-stable
endopermutation $\OQ$-module having an indecomposable direct summand
with vertex $Q$. Set $U=$ $A\tenOQ V$. The following hold.

\smallskip\noindent (i)
As an $\OQ$-module, $U$ is an endopermutation module, and $U$ has a direct
summand isomorphic to $V$.

\smallskip\noindent (ii)
Let $R$ be a subgroup of $Q$. The $A$-module structure on $U$ induces
an $A(\Delta R)$-module structure on  $U'=$ $\Defres^Q_{RC_Q(R)/R}(U)$
extending the $kC_Q(R)$-module structure on $U'$  such that we have an 
isomorphism $(\End_\CO(U))(\Delta R)\cong$ $\End_k(U')$ as algebras and 
as $A(\Delta R)$-$A(\Delta R)$-bimodules.
\end{Proposition}

\begin{proof}
By \cite[5.2]{Liperm},  every indecomposable direct
summand of $A$ as an $\OQ$-$\OQ$-bimodule is isomorphic to
$\OQ \ten_\OR {_\varphi{\OQ}}$ for some subgroup $R$ of $Q$ and
some morphism $\varphi\in$ $\Hom_\CF(R,Q)$, and at least one
summand of $A$ as an $\OQ$-$\OQ$-bimodule is isomorphic to $\OQ$.  
Thus every indecomposable direct
summmand of $U$ is isomorphic to $\Ind^Q_R({_\varphi{V_{\varphi(R)}}})$
for some subgroup $R$ of $Q$ and some $\varphi\in$ $\Hom_\CF(R,Q)$,
where $V_{\varphi(R)}$ is an indecomposable direct summand of vertex
$\varphi(R)$ of $\Res^Q_{\varphi(R)}(V)$, and $V$ is a summand of $U$
as an $\OQ$-module. Since $V$ is $\CF$-stable, we have
${_\varphi(V_{\varphi(R)})}\cong$ $V_R$, which implies that the
restriction to $\OQ$ of $U$ is an endopermutation $\OQ$-module.
Statement (i) follows. For statement (ii) we consider the structural
algebra homomorphism $A\to$ $\End_\CO(U)$ given by the action of $A$
on $U$. This is a homomorphism of interior $Q$-algebras. Applying the
Brauer construction with respect to $\Delta R$, where $R$ is a subgroup
of $Q$, yields a homomorphism of interior $C_Q(R)$-algebras
$A(\Delta R)\to$ $(\End_\CO(U))(\Delta R)$. Since $U$ is an
endopermutation $\OQ$-module, we have $(\End_\CO(U))(\Delta R)\cong$
$\End_k(U')$ as interior $C_Q(R)$-algebras. This yields a homomorphism
$A(\Delta R)\to$ $\End_k(U')$, hence a canonical $A(\Delta R)$-module
structure on $U'$ with the properties as stated.
\end{proof}

Statement (ii) in Proposition \ref{endoinducedsource} is particularly
useful when $Q$ is fully $\CF$-centralised, since in that
case $C_P(Q)$ is a defect group of the unique block $e_Q$
of $kC_G(Q)$ satisfying $\Br_{\Delta Q}(i)e_Q\neq$ $0$,
and the algebras $A(\Delta Q)$ and $kC_G(Q)e_Q$ are
Morita equivalent. Statement (ii) of \ref{endoinducedsource} is
essentially equivalent to a result of Biland; since we will use this
for proving that the Theorems \ref{endopermutationstableMorita}
and \ref{blockendopermutationstableMorita} are equivalent
we state this and sketch a proof for the convenience of the reader.

\begin{Proposition}[{Biland \cite[Theorem 3.15 (i)]{Biland}}] 
\label{blockendoinducedsource}
Let $G$ be a finite group, $b$ a block of $\OG$, $P$ a defect group
of $b$ and $i\in$ $(\OGb)^{\Delta P}$ an almost source idempotent.
Let $Q$ be a subgroup of $P$ and let $V$ be an $\CF$-stable
endopermutation $\OQ$-module having an indecomposable direct summand
with vertex $Q$. Set $X=$ $\OG i\tenOQ V$. Let $R$ be a subgroup of $Q$,
denote by $e_R$ the unique block of $kC_G(R)$ satisfying
$\Br_{\Delta R}(i)e_R\neq$ $0$, and let $\hat e_R$ be the block
of $\CO C_G(R)$ which lifts $e_R$. 
There is a canonical $kC_G(R)e_R$-module $Y_R$ such that we have
an isomorphism $(\End_\CO(\hat e_R Y))(\Delta R)\cong$ $\End_k(Y_R)$
as algebras and as $kC_G(R)e_R$-$kC_G(R)e_R$-bimodules.
\end{Proposition}

\begin{proof}
Applying $\Br_R$ to the canonical algebra homomorphism $\OGb\to$ 
$\End_k(Y)$ and cutting by $e_R$ and $\hat e_R$ yields an algebra
homomorphism $kC_G(R)e_r\to$ $(\End_\CO(\hat e_RY))(\Delta R)$.
In order to show that this is isomorphic to $\End_k(Y_R)$ for some 
module $Y_R$ it suffices to observe that the indecomposable summands of 
$\Res_R(\hat e_R Y)$ with vertex $R$ are all isomorphic. Note that 
$\hat e_R Y=$ $\hat e_R\OG i\tenOQ V$. Any indecomposable direct 
summand of $\hat e_R \OG i$ as an $\OR$-$\OQ$-bimodule with a vertex of 
order at least $|R|$ is isomorphic to ${_\varphi{\OQ}}$ for some group 
homomorphism $\varphi : R\to$ $Q$ induced by conjugation with an element 
in $G$. In view of the fusion stability of $V$, it suffices to show that 
$\varphi$ is a morphism in $\CF$. This is an immediate consequence of 
\ref{Brauerpairfusion}, whence the result.
\end{proof}

As mentioned earlier, there is a technical difference between
the Propositions \ref{endoinducedsource} and \ref{blockendoinducedsource}:
statement (i) in Proposition \ref{endoinducedsource} may not have an
an analogue at the block algebra level, since it is not clear
whether $\hat e_R X$ is an endopermutation $\OR$-module, because the
indecomposable direct summands with vertex strictly contained in $R$
might not be compatible.

\section{Bimodules with fusion-stable endopermutation source}

Throughout this Section we fix the following notation and hypotheses.
Let $G$, $H$ be finite groups, $b$ a block of $\OG$ and
$c$ a block of $\OH$. Suppose that $b$ and $c$ have a
common defect group $P$. Let $i\in$ $(\OGb)^{\Delta P}$
and $j\in$ $(\OHc)^{\Delta P}$ be almost source idempotents.
Set $A=$ $i\OG i$ and $B=$ $j\OH j$.
Suppose that $A$ and $B$ determine the same fusion system $\CF$
on $P$. Let $V$ be an $\CF$-stable indecomposable endopermutation
$\OP$-module with vertex $P$. Whenever expedient, we consider $V$ 
as an $\CO\Delta P$-module through the canonical isomorphism 
$\Delta P\cong$ $P$. Set 
$$U= A\tenOP\Ind^{P\times P}_{\Delta P}(V)\tenOP B\ , $$
$$X = \OG i\tenOP \Ind^{P\times P}_{\Delta P}(V)\tenOP j\OH \ .$$
The $A\tenO B^\op$-module $U$ corresponds to the $\CO(G\times H)$-module
$X$ through the canonical Morita equivalence between $A\tenO B^\op$ and
$\OGb\tenO \OHc^\op$; in particular, there is a canonical bijection
between the isomorphism classes of indecomposable direct summands of
$U$ and of $X$. 
This Section contains some technical statements which involve the tensor 
product of two bimodules. This yields a priori four module structures, 
and keeping track of those is essential - see Brou\'e 
\cite[\S 1]{BroueHigman} for some formal properties of {\it quadrimodules}. 
If the algebras under consideration are group algebras, we play this back 
to two actions via the usual `diagonal' convention: given two finite groups 
$G$, $H$ and two $\OG$-$\OH$-bimodules $S$, $S'$, we consider $S\tenO S'$ 
as an $\OG$-$\OH$-bimodule via the diagonal left action by $G$ and the 
diagonal right action by $H$; explicitly, $x\cdot(s\ten s')\cdot y=$ 
$xsy\ten xs'y$, where $x\in$ $G$, $y\in$ $H$, $s\in$ $S$, and $s'\in$ $S'$. 
This is equivalent to the diagonal $G\times H$-action if we interpret the 
$\OG$-$\OH$-bimodules as $\CO(G\times H)$-modules in the usual way.
The following result is a bimodule version of 
\ref{endoinducedsource}.

\begin{Proposition} \label{endoinducedsourcebimod}
Consider $U$ as an $\CO\Delta P$-module, with $(u,u)\in$ $\Delta P$ 
acting on $U$ by left multiplication with $u$ and right multiplication 
with $u^{-1}$. Then, as an $\CO\Delta P$-module, $U$ is an 
endopermutation module having $V$ as a direct summand, and for any 
subgroup $Q$ of $P$, the $A$-$B$-bimodule structure on $U$ induces an
$A(\Delta Q)$-$B(\Delta Q)$-bimodule structure on
$U'=$ $\Defres^{\Delta P}_{\Delta QC_P(Q)/Q}(U)$ such that we
have an isomorphism of $A(\Delta Q)\tenk B(\Delta Q)^\op$-bimodules
$\End_\CO(U)(\Delta Q)\cong$ $\End_k(U')$.
\end{Proposition}

\begin{proof}
This is the special case of \ref{endoinducedsource}
with $P\times P$, $\CF\times \CF$, $\Delta P$, $\Delta Q$, $A\tenO B^\op$,
instead of $P$, $\CF$, $Q$, $R$, $A$, respectively.
\end{proof}

\begin{Theorem} \label{EndBODeltaQ}
Let $Q$ be a subgroup of $P$, and let $U'$ be the
$A(\Delta Q)$-$B(\Delta Q)$-bimodule from  \ref{endoinducedsourcebimod} 
such that $(\End_\CO(U))(\Delta Q)\cong$ $\End_k(U')$. Then 
$\End_{B^\op}(U)$ is a $\Delta Q$-subalgebra of $\End_\CO(U)$, the 
algebra homomorphism
$$\beta : \End_{B^\op}(U)(\Delta Q) \to \End_\CO(U)(\Delta Q)$$
induced by the inclusion $\End_{B^\op}(U)\subseteq$ $\End_\CO(U)$
is injective, and there is a commutative diagram of algebra homomorphisms
$$\xymatrix{ \End_\CO(U)(\Delta Q) \ar[rr]^{\cong} & & \End_k(U') \\
\End_{B^\op}(U)(\Delta Q) \ar[u]^{\beta} \ar[rr]_{\gamma} & & 
\End_{B(\Delta Q)^\op}(U') \ar[u] }$$
where the right vertical arrow is the obvious inclusion map.
In particular, the algebra homomorphism $\gamma$  is injective.
\end{Theorem}

\begin{proof}
For $\varphi\in$ $\End_\CO(U)$, $y\in$ $Q$, and $u\in$ $U$ we have
$${^{\Delta y}\varphi}(u) = \Delta y \cdot\varphi(\Delta y^{-1}\cdot u) =
y\varphi(y^{-1}uy) y^{-1}\ ,$$
where $\Delta y=$ $(y,y)$. If $\varphi\in$ $\End_{B^\op}(U)$, then
in particular $\varphi$ commutes with the right action by $Q$, and
hence we have  ${^{\Delta y}\varphi}(u) = y\varphi(y^{-1}u)=$
${^{(y,1)}\varphi}(u)$, which shows that ${^{\Delta y}\varphi}$ is
again a $B^\op$-homomorphism. The algebra of $\Delta Q$-fixed points in
$\End_{B^\op}(U)$ is equal to $\End_{\OQ\tenO B^\op}(U)$.
The existence of a commutative diagram as in the statement is formal:
if $\varphi\in$  $\End_{\OQ\tenO B^\op}(U)$, then in particular
$b\cdot\varphi=$ $\varphi\cdot b$ for all $b\in$ $B$, hence for all $b\in$
$B^{\Delta Q}$, and applying $\Br_{\Delta Q}$ yields that the image of
$\varphi$ in $\End_\CO(U)(\Delta Q)$ commutes with the elements in
$B(\Delta Q)$. Since the upper horizontal map is a bimodule isomorphism,
it follows that the image of $\varphi$ in $\End_k(U')$ commutes with the
elements in $B(\Delta Q)$, hence lies in the subalgebra
$\End_{B(\Delta Q)^\op}(U')$.
In order to show that $\beta$ is injective, we first note that this
injectivity  does not make use of the left $A$-module structure of $U$
but only of the left $\OQ$-module structure. Thus we may decompose $U$
by decomposing $A$ as an $\OQ$-$\OP$-bimodule. By
\ref{almostsourcealgebrafusion}, every summand of $A$ as an
$\OQ$-$\OP$-bimodule is of the form $\OQ\tenOR{_\varphi{\OP}}$ for some
subgroup $R$ of $Q$ and some homomorphism $\varphi : R\to$ $P$
belonging to the fusion system $\CF$.
Using the appropriate version of the isomorphism \ref{diagbimodtensorproduct}
of $\OP$-$B$-modules
$\Ind^{P\times P}_{\Delta P}(V)\tenOP B\cong$ $V\tenO B$
it suffices therefore to show that
applying $\Br_{\Delta Q}$ to the inclusion  map
$$\Hom_{B^\op}(\OQ\tenOR{_\varphi{(V\tenO B)}},
\OQ\tenOS{_\psi{(V\tenO B)}}) \subseteq $$
$$\Hom_{\CO}(\OQ\tenOR{_\varphi{(V\tenO B)}},
\OQ\tenOS{_\psi{(V\tenO B)}}) $$
remains injective upon applying $\Br_{\Delta Q}$, where $R$, $S$ are 
subgroups of $Q$ and where $\varphi\in$ $\Hom_\CF(R,P)$, $\psi\in$ 
$\Hom_\CF(S,P)$. If one of $R$, $S$ is a proper subgroup of $Q$, then 
both sides vanish upon applying $\Br_{\Delta Q}$.
Thus it suffices to show that the map
$$\Hom_{B^\op}({_\varphi{(V\tenO B)}},
{_\psi{(V\tenO B)}})(\Delta Q) \to
\Hom_{\CO}({_\varphi{(V\tenO B)}},
{_\psi{(V\tenO B)}})(\Delta Q) $$
is injective, where $\varphi$, $\psi\in$ $\Hom_\CF(Q,P)$.
The summands of ${_\varphi{V}}$, ${_\psi{V}}$ with vertices smaller
than $Q$ yield summands of $V\tenO B$ which vanish on both sides upon
applying $\Br_{\Delta Q}$. The fusion stability of $V$ implies that
indecomposable summands with vertex $Q$ of ${_\varphi{V}}$, ${_\psi{V}}$
are all isomorphic to an indecomposable direct summand $W$ with vertex $Q$
of $\Res^P_Q(V)$. Thus it suffices to show that the map
$$\End_{B^\op}(W\tenO B)(\Delta Q)\to \End_\CO(W\tenO B)(\Delta Q)$$
is injective, where $W$ is an indecomposable direct summand of $\Res^P_Q(V)$
with vertex $Q$. Using the natural adjunction isomorphism
$$\End_\CO(W\tenO B)\cong \Hom_\CO(B, W^*\tenO W\tenO B)$$
it suffices to show that the map
$$\Hom_B(B, W^*\tenO W\tenO B)(\Delta Q) \to 
\Hom_\CO(B,W^*\tenO W\tenO B)(\Delta Q)$$
is injective. Now $W^*\tenO W$ is a direct sum of a trivial $\OQ$-module
$\CO$ and indecomposable  permutation $\OQ$-modules with vertices strictly
smaller than $Q$. Thus it suffices to show that the map
$$\End_{B^\op}(B)(\Delta Q) \to \End_\CO(B)(\Delta Q)$$
is injective. The canonical isomorphism $\End_{B^\op}(B)\cong$ $B$ yields
an isomorphism $\End_{B^\op}(B)(\Delta Q)\cong$ $B(\Delta Q)$.
Since $B$ has a $\Delta Q$-stable $\CO$-basis, it follows that 
$\End_\CO(B)(\Delta Q)\cong$ $\End_k(B(\Delta Q))$.
Using these isomorphisms, the last map is identified with the
structural homomorphism
$$B(\Delta Q) \to \End_k(B(\Delta Q))\ ,$$
which is clearly injective.
\end{proof}

\begin{Lemma} \label{UUstarsummands}
Let $W$ be an indecomposable direct summand of $U\tenB U^*$ or of
$U\tenOP U^*$. Then $W$ is isomorphic to a direct summand of $A\tenOQ A$
for some fully $\CF$-centralised subgroup $Q$ of $P$ such that 
$W(\Delta Q)\neq$ $\{0\}$.
In particular, $U\tenB U^*$ and $U\tenOP U^*$ are $p$-permutation
bimodules.
\end{Lemma}

\begin{proof}
Since $B$ is isomorphic to a direct summand of $B\tenOP B$, it follows
that $U\tenB U^*$ is isomorphic to a direct summand of $U\tenOP U^*$.
Thus it suffices to prove the statement for an indecomposable direct
summnad $W$ of $U\tenOP U^*$.
Using the isomorphisms from \ref{diagbimodtensorproduct},
we get isomorphisms as $\OP$-$\OP$-bimodules
$$\Ind^{P\times P}_{\Delta P}(V^*)\tenOP B\tenOP B\tenOP
\Ind^{P\times P}_{\Delta P}(V)\cong 
V^*\tenO B\tenOP B \tenO V \cong (V\tenO V^*)\tenO (B\tenOP B)\ ,$$
where the right side is to be understood as a tensor product of
two $\OP$-$\OP$-bimodules with the above conventions.
Every indecomposable summand of $B\tenOP B$ as an $\OP$-$\OP$-bimodule
is isomorphic to $\OP \tenOQ {_\varphi{\OP}}\cong$
$\Ind^{P\times P}_{\Delta_\varphi Q}(\CO)$ for some subgroup
$Q$ of $P$ and some morphism $\varphi : Q\to P$ in $\CF$, where
$\Delta_\varphi Q=$ $\{(u,\varphi(u))\ |\ u\in Q\}$. Thus
any indecomposable direct summand of $(V\tenO V^*)\tenO (B\tenOP B)$ is
isomorphic to a direct summand of an $\OP$-$\OP$-bimdoule
of the form $\Ind^{P\times P}_{\Delta_\varphi Q}(V\tenO V^*)$.
The restriction to $\Delta_\varphi Q$ of $V\tenO V^*$ is
a permutation module thanks to the stability of $V$, and hence
the indecomposable direct summands of
$\Ind^{P\times P}_{\Delta_\varphi Q}(V\tenO V^*)$ are
of the form $\OP\tenOR {_\varphi{\OP}}$, where $R$ is a subgroup
of $Q$ and where we use abusively the same letter $\varphi$ for the
restriction of $\varphi$ to any such subgroup.
Thus $W$ is isomorphic to a direct summand of $A\tenOR {_\varphi{A}}$, 
with $R$ and $\varphi$ as before. Set $S=$ $\varphi(R)$.
The indecomposability of $W$ implies
that $W$ is isomorphic to a direct summand of $Ar\tenOR {_\varphi{mA}}$
for some primitive idempotent $r$ in $A^{\Delta R}$ and some primitive
isempotent $m$ in $A^{\Delta S}$. By choosing
$R$ minimal, we may assume that $r$, $m$ belong to local points of $R$ and 
$S$ on $A$, respectively. Let $T$ be a fully $\CF$-centralised subgroup of 
$P$ and let $\psi : T\to$ $R$ be an isomorphism in $\CF$. Then, by
\ref{almostsourcealgebrabimodfusion}, there are primitive idempotents
$n$, $s$ in $A^{\Delta T}$ such that $An\cong$ $Am_\psi$ as $A$-$\OT$-bimodules
and $sA\cong$ ${_{\psi\circ\varphi}{rA}}$ as $\OT$-$A$-bimodules. Thus
$Y$ is isomorphic to a direct summand of $A\tenOT A$. The minimality of $R$, hence
of $T$, implies that $\Delta T$ is a vertex of $\OG i\tenA W\tenB j\OH$, viewed
as an $\CO(G\times H)$-module, and Proposition \ref{vertexsourcealgebra}
implies that a source, which has just shown to be trivial, is a summand of
$W$ restricted to $\Delta T$, which implies that $W(\Delta T)\neq$ $\{0\}$. 
\end{proof}

\begin{Proposition} \label{AtoMMdualsplits}
Let $M$ be an indecomposable direct summand of the $A$-$B$-bimodule $U$.
The following statements are equivalent.

\smallskip\noindent (i)
$A$ is isomorphic to a direct summand of the $A$-$A$-bimodule $M\tenB M^*$.

\smallskip\noindent (ii)
$A$ is isomorphic to a direct summand of the $A$-$A$-bimodule $M\tenOP M^*$.

\smallskip\noindent (iii)
$(M\tenB M^*)(\Delta P)\neq$ $\{0\}$.

\smallskip\noindent (iv)
$(M\tenOP M^*)(\Delta P)\neq$ $\{0\}$.
\end{Proposition}

\begin{proof}
Since $B$ is isomorphic to a direct summand of the $B$-$B$-bimodule
$B\tenOP B$, it follows that $M\tenB M^*$ is isomorphic to a direct summand of
$M\tenOP M^*$. This yields the implications (i) $\Rightarrow$ (ii) and (iii)
$\Rightarrow$ (iv). Since $A(\Delta P)\neq$ $\{0\}$, we trivially have the
implications (i) $\Rightarrow$ (iii) and (ii) $\Rightarrow$ (iv).
Since $M$ is finitely generated projective as a left $A$-module and as a right
$B$-module (hence also as a right $\OP$-module), we have
$M\tenB M^*\cong$ $\End_{B^\op}(M)$, and $M\tenOP M^*\cong$ $\End_{(\OP)^\op}(M)$.
It is well-known that if $A$ is isomorphic to a direct
summand of $M\tenOP M^*$, then the canonical algebra homomorphism
$A\to$ $\End_{(\OP)^\op}(M)$ is split injective as a bimodule homomorphism
(see e.g. \cite[Lemma 4]{KeLiHHbound} for a proof).
This algebra homomorphism factors through the inclusion $\End_{B^\op}(M)\subseteq$
$\End_{(\OP)^\op}(M)$, which implies that the canonical algebra
homomorphism $A\to$ $\End_{B^\op}(M)$ is also split injective as a bimodule
homomorphism. This shows the implication (ii) $\Rightarrow$ (i).
Suppose that (iv) holds. 
Set $Y=$ $\OG i\tenA M$.
Then $Y\tenOP Y^*\cong$ $\OG i\tenA M\tenOP M^* \tenA i\OG$.
It follows from  \ref{UUstarsummands} that $Y\tenOP Y^*$ is a
permutation $\CO(P\times P)$-module on which $\OP$ acts freely on the 
left and on the right and that
$\OP$ is isomorphic to a direct summand of $Y\tenOP Y^*\cong$ 
$\End_{\OP^\op}(Y)$. Thus
$(\End_{\OP^\op}(X))(\Delta P)$ is nonzero projective as a left or
right $kZ(P)$-module. It follows from \cite[Proposition 3.8]{PuigScopes}
that $\OGb$ is isomorphic to a direct summand of $Y\tenOP Y^*$ as an
$\OGb$-$\OGb$-bimodule. Multiplying by $i$ on the left and right implies
that $A$ is isomorphic to a direct summand of $M\tenOP M^*$. 
Thus (iv) implies (ii), completing the proof.
\end{proof}

\begin{Proposition} \label{BtoMdualMsplits}
Let $M$ be an indecomposable direct summand of the $A$-$B$-bimodule $U$.
Then $A$ is isomorphic to a direct summand of the $A$-$A$-bimodule
$M\tenB M^*$ if and only if $B$ is isomorphic to a direct summand of
the $B$-$B$-bimodule $M^*\tenA M$. 
\end{Proposition}

\begin{proof}
Suppose that $A$ is isomorphic to a direct summand of $M\tenB M^*$, but
that $B$ is not isomorphic to a direct summand of
$M^*\tenA M$. It follows from Lemma \ref{UUstarsummands}, applied
to $B$, $A$, $V^*$ instead of $A$, $B$, $V$, respectively, that
$M^*\tenB M$ is a direct sum of summands of bimodules of the form
$B\tenOQ B$, with $Q$ running over a family of proper subgroups of $P$.
Thus $M\tenB M^*\tenA M\tenB M^*$ is a direct sum of summands of
bimodules of the form $M\tenOQ M^*$, with $Q$ running over a family
of proper subgroups of $P$. In particular, we have
$(M\tenB M^*\tenA M\tenB M^*)(\Delta P)=$ $\{0\}$.
But $A\cong$ $A\tenA A$ is a summand
of  $M\tenB M^*\tenA M\tenB M^*$, hence
$(M\tenB M^*\tenA M\tenB M^*)(\Delta P)\neq$ $\{0\}$. This contradiction
shows that $B$ is isomorphic to a direct summand of $M^*\tenA M$.
Exchanging the roles of $A$ and $B$ yields the converse.
\end{proof}

In particular, if the equivalent statements in Proposition 
\ref{BtoMdualMsplits} hold, then the algebras $A$, $B$ are separably 
equivalent (cf. \cite[Definition 3.1]{LinHecke}).

\begin{Proposition} \label{blockAtoMMdualsplits}
Let $M$ be an indecomposable direct summand of the $\OGb$-$\OHc$-bimodule
$X$. The following are equivalent.

\smallskip\noindent (i)
$\OGb$ is isomorphic to a direct summand of the $\OGb$-$\OGb$-bimodule
$M\tenOHc M^*$.

\smallskip\noindent (ii)
$\OHc$ is isomorphic to a direct summand of the $\OHc$-$\OHc$-bimodule
$M^*\tenOGb M$.

\smallskip\noindent (iii)
$M$ has vertex $\Delta P$. 

\smallskip\noindent
If these equivalent conditions hold, then $V$ is an $\CO\Delta P$-source 
of $M$.
\end{Proposition}

\begin{proof}
Set $A=$ $i\OG i$ and $B=$ $j\OH j$.
The equivalence of (i) and (ii) is a reformulation of
\ref{BtoMdualMsplits} at the level of block algebras, via the
standard Morita equivalences between block algebras and almost
source  algebras. The bimodule $M$ has a vertex $\Delta Q$ contained in
$\Delta P$, for some subgroup $Q$ of $P$. If this vertex is
smaller than $\Delta P$, then 
$(M\tenOHc M^*)(\Delta P)=$ $\{0\}$, so also $(iMj\tenB jM^*i)(\Delta P)=$
$\{0\}$. Thus \ref{AtoMMdualsplits} implies that $iMj\tenB jM^*i$ has
no summand isomorphic to $B$, hence $M\tenOHc M^*$ has no summand
isomorphic to $\OHc$. This shows that (ii) implies (iii). Suppose that
$\Delta P$ is a vertex of $M$. Then clearly $V$ is a source of $M$. 
By \ref{vertexsourcealgebra}, 
$M$ has a vertex source pair $(P',V')$ such that $P'\subseteq$ $P\times P$
and such that $V'$ is a direct summand of $iMj$ as an $\OP'$-module.
It follows that as an $\CO(P\times P)$-module,  $iMj$ has an indecomposable
direct summand $W$ with vertex $P'$ and source $V'$. Green's indecomposability
theorem implies that $W\cong$ $\Ind^{P\times P}_{P'}(V')$ is a summand
of $iMj$ as an $\OP$-$\OP$-bimodule, hence of
$A\tenOP\Ind^{P\times P}_{\Delta P}(V)\tenOP B$.
Using the bimodule structure of $A$ and $B$, it follows that
$P'$ is a `twisted' diagonal subgroup of the form
$\{(\varphi(u), \psi(u))\ |\ u\in P$ for some
$\varphi$, $\psi\in$ $\Aut_\CF(P)$. Since $A\cong$ ${_\varphi{A}}$ as
$\OP$-$A$-bimodules and $B\cong$ $B_\psi$ as $B$-$\OP$-bimodules,
it follows that $iMj$ has a direct summand isomorphic to
to $\Ind^{P\times P}_{\Delta P}(V')$, and then $V'\cong$ $V$ by the
stability of $V$.
But then $iMj\tenOP jM^*i$ has a summand isomorphic to
$\Ind_{\Delta P}^{P\times P}(V)\tenOP \Ind^{P\times P}_{\Delta P}(V^*)\cong$
$\Ind^{P\times P}_{\Delta P}(V\tenO V^*)$. Since $V\tenO V^*$ has a
trivial summand, it follows that $iMj\tenOP jM^*i$ has a summand isomorphic to
$\OP$, which implies that $(iMj\tenOP jM^*i)(\Delta P)\neq$ $\{0\}$.
Proposition \ref{AtoMMdualsplits} implies that $A$ is isomorphic to
a direct summand of $iMj\tenB jM^*i$, and hence $\OGb$ is isomorphic to
a direct summand of $M\tenOHc M^*$, completing the proof.
\end{proof}

\section{Proof of Theorem  \ref{endopermutationstableMorita} and of
Theorem \ref{blockendopermutationstableMorita}}

\begin{proof} [{Proof of Theorem \ref{endopermutationstableMorita}}]
We use the notation and hypotheses from Theorem
\ref{endopermutationstableMorita}.
Since $(M\tenB M^*)(\Delta P)\neq$ $\{0\}$, it follows from
\ref{AtoMMdualsplits} that $M\tenB M^*\cong$ $A\oplus X$ for some
$A$-$A$-bimodule $X$ with the property that every indecomposable
direct summand of $X$ is isomorphic to a direct summand of
$A\tenOQ A$ for some fully $\CF$-centralised subgroup $Q$ of $P$.
In what follows we use the canonical isomorphism $M\tenB M^*\cong$
$\End_B(M)$ and analogous versions. By \ref{EndBODeltaQ},
for any subgroup $Q$ of $P$ we have an injective algebra homomorphism
$$\End_{B^\op}(M)(\Delta Q) \to \End_{B(\Delta Q)^\op}(M_Q) $$
The left term is isomorphic to  $A(\Delta Q)\oplus X(\Delta Q)$. If $Q$
is nontrivial and fully $\CF$-centralised, then the right term is
isomorphic to $A(\Delta Q)$ by the assumptions on $M_Q$. This forces
$X(\Delta Q)=$ $\{0\}$ for any nontrivial fully $\CF$-centralised
subgroup $Q$ of $P$. It follows from \ref{UUstarsummands} that
$X$ is projective as an $A$-$A$-bimodule.
Similarly, \ref{BtoMdualMsplits} implies that $M^*\tenA M\cong$ $B\oplus Y$
for some $B$-$B$-bimodule $Y$, and the same argument with the roles of
$A$ and $B$ exchanged shows that $Y$ is projective.
\end{proof}

\begin{proof} [{Proof of Theorem \ref{blockendopermutationstableMorita}}]
By multiplying the involved bimodules with almost source idempotents, it
follows using the block algebra versions \ref{blockendoinducedsource}
and \ref{blockAtoMMdualsplits} of \ref{AtoMMdualsplits}
and of \ref{endoinducedsource}, respectively, 
that Theorem \ref{blockendopermutationstableMorita}
is equivalent to Theorem \ref{endopermutationstableMorita}.
\end{proof}

\begin{Remark} \label{endosplitRemark}
We sketch a proof of Theorem \ref{endopermutationstableMorita} under
the additional assumption that the endopermutation $\OP$-module $V$ 
has an $\CF$-stable $p$-permutation resolution $Y_V$
(cf. \cite[\S 7]{Ricksplendid}). That is, $Y_V$ is a bounded complex
of permutation $\OP$-modules such that the complex $Y_V\tenO Y_V^*$ is
split as a complex of $\OP$-modules with respect to the diagonal
action of $P$, and such that $Y_V\tenO Y_V^*$ has homology concentrated
in degree zero and isomorphic to $V\tenO V^*$. The $\CF$-stability means
that for any subgroup $Q$ of $P$ and any morphism $\varphi : Q\to$ $P$ in
$\CF$ the indecomposable summands of $\Res^P_Q(Y_V)$ and $\Res_{\varphi}(Y_V)$
with vertex $Q$ (as complexes) are isomorphic (this is slightly weaker than
the condition stated in \cite[Theorem 1.3]{LinTR}). The proof of
\cite[Theorem 1.3]{LinTR} yields an indecomposable direct summand $Y$ of
the complex $A\tenOP\Ind^{P\times P}_{\Delta P}(Y_V)\tenOP B$ such that
$Y\tenB Y^*$ is split with homology concentrated in degree zero isomorphic
to $M\tenB M^*$; similarly for $Y^*\tenA Y$. Note that $Y$ is splendid in
the sense of \cite[1.10]{LinTR} or \cite[1.1]{Lisplendid}.
It follows from \cite[\S 7.3]{Ricksplendid} that if $Q$ is a fully
$\CF$-centralised subgroup of $P$, then $Y(\Delta Q)$ is a bounded
complex of $A(\Delta Q)$-$B(\Delta Q)$-bimodules with homology concentrated
in a single degree and isomorphic to a bimodule $M_Q$ as in the
statement of the Theorem.  It follows from
\cite[Proposition 2.4]{LinTR} or \cite[Theorem 9.2]{Liperm} that
for any fully $\CF$-centralised subgroup $Q$ of $P$ we have
$(Y\tenB Y^*)(\Delta Q)\cong$ $Y(\Delta Q)\ten_{B(\Delta Q)} Y(\Delta Q)^*$
and this complex is again split with homology concentrated in degree
zero isomorphic to $M_Q\ten_{B(\Delta Q)} M_Q^*$. Thus if $M_Q$ induces
a Morita equivalence, then $A(\Delta Q)$ $M_Q\ten_{B(\Delta Q)} M_Q^*$. 
Therefore, if $M_Q$ induces a Morita equivalence for all nontrivial fully
$\CF$-centralised subgroups $Q$ of $P$, then $Y(\Delta Q)$ induces
in particular a derived equivalence for all such $Q$, and hence, by
a result of Rouquier (see \cite[Appendix]{Liperm} for a proof)
the complex $Y$ induces a stable equivalence. This implies that
$M$ induces a stable equivalence, providing thus an alternative
proof of Theorem \ref{endopermutationstableMorita}.
\end{Remark}

\section{Appendix}

In the proof of Proposition \ref{AtoMMdualsplits} we have made use of 
\cite[Proposition 3.8]{PuigScopes}. The purpose of this section is to
give a proof of a slightly more general result in this direction.
We use without further comment the following standard properties of 
$p$-permutation modules: if $U$ is an indecomposable $\OG$-module with 
vertex $P$ and trivial source, then the $kN_G(P)$-module $U(P)$ is the 
Green correspondent of $k\tenO U$, and we have a canonical algebra
isomorphism $(\End_\CO(U))(P)\cong$ $\End_k(U(P))$. Moreover, as a 
$kN_G(P)/P$-module, $U(P)$ is the multiplicity module of $U$; in 
particular, $U(P)$ is projective indecomposable as a $kN_G(P)/P$-module. 
Any $p$-permutation $kG$-module lifts uniquely, up to isomorphism, to a 
$p$-permutation $\OG$-module. In particular, the isomorphism class of an 
indecomposable $\OG$-module $U$ with vertex $P$ and trivial source is 
uniquely determined by the isomorphism class of the projective 
indecomposable $kN_G(P)/P$-module $U(P)$. See e. g. \cite[\S 27]{Thev}
for an expository account on $p$-permutation modules with further references.
The following result is well-known (we include a proof for the convenience
of the reader):

\begin{Proposition} \label{splitinjective1}
Let $G$ be a finite group, $P$ a $p$-subgroup, $U$ an indecomposable
$\OG$-module with vertex $P$ and trivial source $\CO$, and let $M$ be an
$\OG$-module such that $\Res^G_P(M)$ is a permutation $\OP$-module.
Set $N=$ $N_G(P)/P$. Let $\alpha : U\to$ $M$ be a
homomorphism of $\OG$-modules. The following are equivalent.

\smallskip\noindent (i)
The $\OG$-homomorphism $\alpha : U\to$ $M$ is split injective.

\smallskip\noindent (ii)
The $kN$-homomorphism $\alpha(P) : U(P)\to$ $M(P)$ is injective.
\end{Proposition}

\begin{proof}
The implication (i) $\Rightarrow$ (ii) is trivial. Suppose that (ii) 
holds. Then $\alpha(P) : U(P)\to$ $M(P)$ is split injective
as a $kN$-homomorphism because $U(P)$ is projective, hence injective,
as a $kN$-module. Using that $\soc(U(P))$ is simple it follows that
$M$ has an indecomposable direct summand 
$M'$ such that the induced map $\beta(P) : U(P)\to$ $M'(P)$ is still
split injective, where $\beta$ is the composition of $\alpha$ 
followed by the projection from $M$ onto $M'$. The Brauer homomorphism
applied to the algebra $\End_\CO(M')$ maps $\End_\CO(M')_P^G$ onto
$(\End_k(M'))(P)_1^N\cong$ $\End_k(M'(P))_1^N$ (cf. \cite[(27.5)]{Thev}). 
The summand of $M'(P)$ isomorphic to $U(P)$ corresponds to a primitive 
idempotent in  $\End_k(M'(P))_1^N$, hence lifts to a primitive idempotent in
$\End_\CO(M')^G_P$. Since $M'$ is indecomposable, this idempotent
is $\Id_{M'}$, and hence, by Higman's criterion, $M'$ has $P$ as
a vertex. But then $M'$ has a trivial source, and so $M'(P)$ is
indecomposable as a $kN$-module, hence isomorphic to $U(P)$.
By the Green correspondence this implies $U\cong$ $M'$. Composing
$\beta$ with the inverse of this isomorphism yields an endomorphism
$\gamma$ of $U$ which induces an automorphism on $U(P)$. Since 
$\End_\OG(U)$ is local, this implies that $\gamma$ is an automorphism
of $U$, and hence that $\beta : U\to$ $M'$ is an isomorphism. It
follows that $\alpha$ is split injective, whence the
implication (ii) $\Rightarrow$ (i).
\end{proof}

\begin{Proposition} \label{splitinjective2}
Let $G$ be a finite group, $P$ a $p$-subgroup, $U$ an indecomposable
$\OG$-module with vertex $P$ and trivial source, and let $M$ be an
$\OG$-module such that $\Res^G_P(M)$ is a permutation $\OP$-module.
Set $N=$ $N_G(P)/P$. 
Suppose that $N$ has a normal $p$-subgroup $Z$
such that the $kN$-module $k\ten_{kZ} U(P)$ is simple and such that
$M(P)$ is projective as a $kZ$-module. Let $\alpha : U\to$ $M$ be a
homomorphism of $\OG$-modules. The following are equivalent.

\smallskip\noindent (i)
The $\OG$-homomorphism $\alpha : U\to$ $M$ is split injective.

\smallskip\noindent (ii)
There is a nonzero direct summand $W$ of $\Res^N_Z(U(P))$ 
such that the $kZ$-homomorphism  $\alpha(P)|_W : W\to$ 
$\Res^N_Z(M(P))$ is injective.
\end{Proposition}

\begin{proof}
The implications (i) $\Rightarrow$ (ii) is trivial.  Suppose that (ii) 
holds. Since $U(P)$ is projective  indecomposable as a $kN$-module, it 
has a simple socle and a simple top, 
and these are isomorphic. The restriction of $U(P)$ to $kZ$ remains 
projective, and hence the module $k\ten_{kZ} U(P)$ has the same dimension 
as the submodule $U(P)^{Z}$ of $Z$-fixed points in $U(P)$. Since $Z$ is 
normal in $N$, it follows that $U(P)^{Z}$ is a $kN$-submodule of $U(P)$,
hence that $U(P)^Z$ contains the simple socle of $U(P)$.
Since $k\ten_{kZ} U(P)$ is assumed to be simple, hence isomorphic to the
top and bottom composition factor of $U(P)$, it follows that $U(P)^Z$ 
is equal to the socle $\soc(U(P))$ of $U(P)$ as a $kN$-module. By the 
assumption (ii), the kernel of the map $U(P)\to$ $M(P)$
does not contain $U(P)^Z$. Since the socle of $U(P)$ as a $kN$-module
is simple, this implies that $\alpha(P) : U(P)\to$ $M(P)$ is injective,
and hence that $\alpha$ is split injective by Proposition
\ref{splitinjective1}. This shows the implication (ii) $\Rightarrow$ (i). 
\end{proof}

\begin{Corollary}[{Puig \cite[Proposition 3.8]{PuigScopes}}]
Let $G$ be a finite group, $b$ a block of $\OG$, $P$ a defect group
of $b$, and $A$ an interior $G$-algebra. Suppose that the conjugation
action of $P$ on $A$ stabilises an $\CO$-basis of $A$, and that
$\Br_{\Delta P}(b) \cdot A(\Delta P)\cdot \Br_{\Delta P}(b)$ is
projective as a left or right $kZ(P)$-module. Then
the map $\alpha : \OGb\to$ $A$ induced by the structural homomorphism
$G\to$ $A^\times$ is split injective as a homomorphism of
$\OGb$-$\OGb$-bimodules.
\end{Corollary}

\begin{proof}
After replacing $A$ by $b\cdot A\cdot b$ we may assume that 
$A(\Delta P)$ is projective as a left or right $kZ(P)$-module.
As an $\CO(G\times G)$-module, $\OGb$ has vertex $\Delta P$ and trivial
source. By the assumptions, $A$ is a permutation $\CO\Delta P$-module.
We have $N_{G\times G}(\Delta P)=$ 
$(C_G(P)\times C_G(P))\cdot N_{\Delta G}(\Delta P)$. Set $N=$
$N_{G\times G}(\Delta P)/\Delta P$. Denote by $Z$ the image of
$Z(P)\times \{1\}$ in $N$; this is equal to the image of $\{1\}\times Z(P)$,
normal in $N$, and canonically isomorphic to $Z(P)$. Consider the
induced map $\alpha(\Delta P) : kC_G(P)\Br_{\Delta P}(b)\to$ $A(\Delta P)$.
If $e$ is a block of $kC_G(P)$ occurring in $\Br_{\Delta P}(b)$,
then $kC_G(P)/Z(P)\bar e$ is a matrix algebra, where $\bar e$ is the
canonical image of $e$ in $kC_G(P)/Z(P)$. (We use here again our assumption that
$k$ is large enough.) Thus $kC_G(P)/Z(P)\bar e$ is
simple as a module over $k(C_G(P)\times C_G(P))$. Since the blocks $e$
arising in this way are permuted transitively by $N_G(P)$, it follows 
that $k\ten_{kZ} kC_G(P)\Br_{\Delta P}(b)\cong$ $kC_G(P)/Z(P)c$ is a 
simple $kN$-module, where $c$ is the image of $\Br_{\Delta P}(b)$ in 
$kC_G(P)/Z(P)$, or equivalently, $c$ is the sum of the $\bar e$ as above.
By the assumptions, $A(\Delta P)$ is projective as a
$kZ$-module, and hence the obvious composition of algebra
homomorphisms $kZ(P)\to$ $kC_G(P)\Br_{\Delta P}\to$ $A(\Delta P)$ is
injective. 
Thus $kC_G(P)\Br_{\Delta P}(b)$ has a summand isomorphic to $kZ$, as
a $kZ$-module, which is mapped injectively into $A(\Delta P)$ by
$\alpha(\Delta P)$. The result follows from the implication
(ii) $\Rightarrow$ (i) in Proposition \ref{splitinjective2}.
\end{proof}

\end{document}